\documentclass[11pt, reqno]{amsart}

\usepackage[top=1in,bottom=1in,left=1.2in,right=1.2in]{geometry}
\usepackage{amsmath,amssymb,amsthm,amsfonts,enumerate, mathrsfs, mathtools, bm}
\usepackage{appendix}
\usepackage{relsize}
\usepackage{stmaryrd}
\usepackage{esint}
\usepackage{xcolor}
\usepackage{graphicx}
\usepackage{natbib}
\usepackage{makecell,float}
\usepackage{multirow}

\usepackage[
colorlinks=true,
linkcolor=blue,
filecolor=blue,
urlcolor=red,
citecolor=magenta]{hyperref}
\usepackage[nameinlink]{cleveref}

\newcounter{counterConstant} 
\newcommand{\const}[1]{
	\addtocounter{counterConstant}{1}
	\edef#1{\arabic{counterConstant}}
}

\usepackage[normalem]{ulem}

\numberwithin{equation}{section}

\newtheorem{theorem}{Theorem}[section]
\newtheorem{lemma}[theorem]{Lemma}
\newtheorem{remark}[theorem]{Remark}
\newtheorem{definition}[theorem]{Definition}
\newtheorem{proposition}[theorem]{Proposition}

\newtheorem{assumption}{Assumption}

\def\cB{{\mathcal B}}

\def\cX{{\mathcal X}}

\def\mN{{\mathbb N}}

\def\mR{{\mathbb R}}

\def\mZ{{\mathbb Z}}

\def\sF{{\mathscr F}}

\def\sS{{\mathscr S}}

\def\geq{\geqslant}
\def\leq{\leqslant}

\def\1{{\mathbf{1}}}

\def\d{\text{\rm{d}}}
\def\e{\mathrm{e}}
\def\eps{\varepsilon}

\usepackage{fancyhdr}
{} 

\allowdisplaybreaks

\begin{document}
	
\title{Stein's Method for Convergence Rates of Invariant Measures in the Nonlocal-to-Local Limit}

\author{Mingyan Wu and Guohuan Zhao}  

\address{School of Mathematical Sciences, Xiamen University, Xiamen, Fujian 361005, P. R. China}
\email{mingyanwu.math@xmu.edu.cn; mingyanwu.math@gmail.com}

\address{ State Key Laboratory of Mathematical Sciences, Academy of Mathematics and Systems Science, CAS, Beijing, 100190, P. R. China}
\email{gzhao@amss.ac.cn}

\thanks{The research of Mingyan Wu is supported by Innovative Youth Science Fund Project of Fujian Provincial Natural Science Foundation (Grant No. 2026J008016) and National Natural Science Foundation of China grants (No. 12201227). The research of Guohuan Zhao is supported by the National Key Research and Development Program of China (No. 2024YFA1013503) and the National Natural Science Foundation of China grants (No. 12271352).}
    
\begin{abstract}
We utilize Stein's method to establish quantitative bounds on the total variation distance between the invariant measure of a drifted nonlocal Markov operator and that of its local counterpart under minimal assumptions on the drifts. The main ingredient is a reduction via Stein's method that transforms the original problem into analyzing growth estimates for solutions to a nonlocal Poisson equation and decay estimates for the invariant measure of the local operator. 
\end{abstract}

\maketitle



\tableofcontents

\section{Introduction} 

The asymptotic recovery of classical differential operators from their nonlocal counterparts is a fundamental theme in analysis. A celebrated result in this direction is the Bourgain-Brezis-Mironescu (BBM) formula \cite{BBM2001another}, which identifies the $O(1-s)$ rate at which fractional Sobolev norms recover classical ones as $s \uparrow 1$. This convergence of functional spaces is closely reflected at the operator level by the fractional Laplacian $-(-\Delta)^{\alpha/2}$ with the correspondence $s = \alpha/2$.  Extensive theoretical and numerical research has demonstrated that solutions to fractional systems consistently recover their local counterparts as $\alpha \to 2$, see for instance \cite{BHS2018Poisson, FJS2019Fractional, Jarrin2024From, HO2014Numerical}

Beyond the functional-analytic framework, this transition carries a deep probabilistic interpretation. Since the fractional Laplacian $-(-\Delta)^{\alpha/2}$ serves as the generator for an $\alpha$-stable process, the analytical limit $\alpha \uparrow 2$ corresponds precisely to the convergence of jump-type noise toward Gaussian diffusion. This naturally raises the question of whether the long-term statistical behavior of these processes--characterized by their invariant measures--inherits a similar asymptotic profile. While the convergence rate of these probabilistic objects cannot be directly deduced from the BBM formula, the $O(1-s)$ rate (with $s = \alpha/2$) serves as a heuristic, suggesting a corresponding $O(2-\alpha)$ scaling. 

Very recently, Deng, Schilling and Xu \cite{DSX2025optimal} considered the following setting: let $b: \mR^d \to \mR^d$ be a locally bounded vector field, and define the infinitesimal generators
\begin{equation*}
    L=\Delta+b\cdot\nabla \quad \text{and} \quad L_{\alpha} = -(-\Delta)^{\frac{\alpha}{2}}+b\cdot\nabla, \quad \alpha\in (0,2),
\end{equation*}
whose associated stochastic processes are given by the following SDEs, respectively:  
\begin{equation*}
    \d X_t = b(X_t) \d t + \sqrt{2} \d W_t, \quad \d X_t^\alpha = b(X_t^\alpha) \d t + \d Z_t^\alpha,
\end{equation*}
where $W_t$ is a standard Brownian motion and $Z_t^\alpha$ is a symmetric $\alpha$-stable process. Under suitable dissipativity conditions, both dynamics admit unique invariant probability measures, denoted by $\pi$ and $\pi_\alpha$, respectively. The authors established that $\pi_\alpha$ converges to $\pi$ at a rate of $O(2-\alpha)$. Complementing this, more recently, Liu, Ren and the first named author of this manuscript investigated the non-Gaussian to Gaussian transition from a more refined heat kernel perspective in \cite{LRW2026non}. They derived pointwise estimates for the difference between the heat kernels of $L_\alpha$ and $L$, showing how the heavy-tailed stable density transforms into the Gaussian density as $\alpha \uparrow 2$. Their work also confirmed the $O(2-\alpha)$ convergence rate for the corresponding invariant measures; further related results can be found in the references therein.

It should be emphasized that the existence and uniqueness of invariant measures can be established under mild conditions, requiring neither regularity assumptions nor growth constraints on the drift $b$ (cf. \cite[Theorem 4.1.6]{BKRS2015fokker} and \cite[Section 3.4.4]{Kul2018ergodic}).
However, quantifying the convergence rates of these measures often necessitates stronger technical assumptions. For instance, the approach in \cite{DSX2025optimal} relies on Malliavin calculus, which requires the drift field to be sufficiently smooth (up to the third-order derivatives). \cite{LRW2026non} further relaxed these requirements via a parametrix-based heat kernel method, establishing convergence under Hölder continuous drifts with dissipative linear growth by exploiting the classical parametrix construction of fundamental solutions pioneered by Levi \cite{Levi1907Sulle} and recently further developed in \cite{Menozzi2021Density, Menozzi2022Heat}. 

In contrast to the above approaches, which all impose regularity and linear growth conditions on the drift field $b$, the main purpose of this note is to study the convergence rates of invariant measures for nonlocal operators $L_\alpha$ to their local counterparts $L$ in total variation distance, under \textit{minimal} assumptions on $b$. Specifically, we do not require any regularity of the drift $b$ and allow it to have exponential-type growth (see Theorem \ref{thm:main}). 

\subsection{Assumption and Main Result}

\const{\Kd}
\const{\Kg}

For $\alpha\in (0,2)$, recall that the fractional Laplacian is given by 
\begin{equation*}
    (-\Delta)^{\frac{\alpha}{2}} f(x) = c_{d}(\alpha) \mathrm{P.V.} \int_{\mR^d} \frac{\left(f(x)-f(x-y)\right)}{|y|^{d+\alpha}} \, \d y, \quad f\in C_b^2(\mR^d), 
\end{equation*}
where  
\[
c_{d}(\alpha):= \pi^{-d/2} 2^\alpha \frac{\Gamma((d+\alpha)/2)}{|\Gamma(-\alpha/2)|}.
\]
The fractional Laplacian can also be defined via Fourier transform with the multiplier $|\xi|^{\alpha}$, and the two definitions are consistent as 
\begin{equation} \label{eq:frac_lap}
    c_{d}(\alpha) \mathop{\mathrm{P.V.}} \int_{\mR^d} (1 - \e^{-iy\cdot\xi})  |y|^{-d-\alpha} \, \d y = |\xi|^\alpha,
\end{equation}
where $\mathop{\mathrm{P.V.}}$ denotes the Cauchy principal value.

Our assumption on $b$ is as follows:  
\begin{assumption}[Dissipativity and Growth]\label{aspt}
Let $b: \mR^d \to \mR^d, (d \geq 1)$, be a locally bounded vector field. For some strictly positive constants $\kappa, \theta, K_{\Kd}, K_{\Kg}>0$, it holds that 
\begin{equation}\label{eq:dispt}
    b(x) \cdot x \leq -\kappa |x|^{\theta} + K_{\Kd}, \quad x\in \mR^d
\end{equation}
and 
\begin{equation}\label{eq:growth}
    |b(x)| \leq K_{\Kg}\exp\left(\frac{\kappa}{100 d \theta}|x|^\theta\right), \quad x\in \mR^d. 
\end{equation}
\end{assumption}

For simplicity of notation, we introduce the parameter set
\[
\Theta := (d, \kappa, \theta, K_{\Kd}, K_{\Kg}).
\]
It is well known that dissipative condition \eqref{eq:dispt} and local boundedness ensure the existence and uniqueness of the invariant measures $\pi$ and $\pi_\alpha$ for the operators $L$ and $L_\alpha$, respectively (see, e.g., \cite[Theorem 4.1.6]{BKRS2015fokker} and \cite[Section 3.4.4]{Kul2018ergodic}). Furthermore, we have the following main result.

\begin{theorem}\label{thm:main}
    Under Assumption \ref{aspt}, it holds that 
    \begin{equation*}
        \|\pi-\pi_\alpha\|_{\mathrm{TV}} \leq C (2-\alpha), 
    \end{equation*}
    where the constant $C>0$ depends only on $\Theta$, and $\| \cdot \|_{\mathrm{TV}}$ defined by
\[
\|\pi-\pi_\alpha\|_{\mathrm{TV}}:= \sup_{f \in \cB_b;\|f\|_{L^\infty}\leq 1} \langle f, \pi-\pi_\alpha \rangle,
\]
denotes the total variation distance.
\end{theorem}

\begin{remark}\label{rmk:main} 
\begin{enumerate}[(i)]
    \item Due to\cite[Proposition 5.1]{DSX2025optimal}, 
    $2-\alpha$ is the optimal rate in the sense of total variation distance.
    \item The dissipative condition \eqref{eq:dispt} ensures the existence, uniqueness, and tail control of the invariant measure $\pi$ and the growth control of the solution to the Poisson equation \eqref{eq:poisson} (see Lemma \ref{lem:int} and Proposition \ref{prop:poisson} (a) below). We note that a stronger requirement, $\theta \geq 2$, was assumed in both \cite{DSX2025optimal} and \cite{LRW2026non}. 
    \item The growth condition \eqref{eq:growth} will be used to control the tail of the nonlocal regularity of $\pi$ and the growth of the nonlocal derivative of the solution to the Poisson equation \eqref{eq:poisson} (see Propositions \ref{prop:IM} and \ref{prop:poisson}(b)). Compared to the linear growth condition required in \cite{DSX2025optimal} and \cite{LRW2026non}, the assumption \eqref{eq:growth} considered here is more general. 
    \item While this paper focuses on the total variation distance for simplicity, the underlying approach is applicable to other weighted variation distances. Furthermore, the method can be extended to more complex models, such as equations with non-constant diffusion coefficients or driven by more general stable processes.
\end{enumerate}
\end{remark}

\subsection{Methodology}\label{sec:MC}

The core of our approach lies in a novel representation of the duality pairing $\langle f, \pi - \pi_\alpha \rangle$ through the lens of a specific Poisson equation and a regularizing Fourier multiplier. 

\medskip\noindent
\paragraph{\em Step 1: Apply Stein's equation} Our estimates rely on the generator approach to Stein's method, originally pioneered by Barbour \cite{Barbour1990stein} (see also \cite{Fang2019Multivariate}). Precisely, we consider the Poisson equation (or Stein's equation):
\begin{equation}\label{eq:poisson}
    L_\alpha u = f-\langle f, \pi_\alpha \rangle, \quad f \in L^\infty(\mathbb{R}^d).
\end{equation}
The key observation for proving Theorem \ref{thm:main} is the following identity, obtained by a formal derivation:
\begin{equation*}
    \begin{aligned}
        \langle f, \pi-\pi_\alpha \rangle = \langle L_\alpha u, \pi\rangle = \langle (L_\alpha-L)u, \pi \rangle =\langle [-\Delta-(-\Delta)^{\frac{\alpha}{2}}]u, \pi \rangle. 
    \end{aligned}
\end{equation*}

However, to utilize this equation, we need to establish global nonlocal regularity estimates for the solution $u$ to the Poisson equation, thereby ensuring its integrability with respect to $\pi$. However, without regularity assumptions on the drift $b$, $u$ is at best in $C^\alpha_{\rm loc}(\mR^d)$, which is insufficient to control the terms on the left-hand side of the above equation. 

\medskip\noindent
\paragraph{\em Step 2: Construct $T,T^{-1}$} To overcome the technical difficulty stated in {\em Step 1}, we introduce an invertible Fourier multiplier operator $T$, similar to $(I-\Delta)^{1/3}$ with a kernel exhibiting rapid decay, specifically faster than $\e^{-|x|^{2\theta}}$ as $|x| \to \infty$, and rewrite the above identity as 
\[
\langle f, \pi-\pi_\alpha \rangle=\langle [-\Delta-(-\Delta)^{\frac{\alpha}{2}}]T^{-1}u, T\pi \rangle. 
\]
The construction of the operator $T$, presented in Section \ref{sec:T_A}, is specifically designed to serve a dual purpose: to ensure that $T^{-1}u$ possesses regularity exceeding the second order, while simultaneously maintaining the strong decay of $T\pi(x)$ at infinity. It is worth noting that a standard choice such as the Bessel potential $T = (I-\Delta)^{1/3}$ would fail to guarantee the latter property. However, we can obtain the desired $T$ through cutoff and scaling of the kernel of $(I-\Delta)^{1/3}$; see Section \ref{sec:T} for details.

\medskip\noindent
\paragraph{\em Step 3: Represent $-\Delta - (-\Delta)^{\alpha/2}$}
To intuitively see why the factor $2-\alpha$ appears, we formally differentiate \eqref{eq:frac_lap} with respect to $\alpha$, which yields  
\begin{equation}\label{eq:diff_frac}
    \begin{aligned}
        \frac{\d}{\d s} |\xi|^s & =|\xi|^s \log|\xi| \\
        =\int_{\mR^d} (1 - \cos(y \cdot \xi))  \Big( c'_{d}(s)  & |y|^{-d-s} - c_{d}(s) |y|^{-d-s} \log|y| \Big) \d y.
    \end{aligned}
\end{equation}
Let $\log(\sqrt{-\Delta})(-\Delta)^{\frac{s}{2}}$ be the Fourier multiplier operator with symbol $|\xi|^s \log|\xi|$. Then its integral kernel is given by     
\[
K_{\log(\sqrt{-\Delta})(-\Delta)^{\frac{s}{2}}} =c'_{d}(s) |y|^{-d-s} - c_{d}(s) |y|^{-d-s} \log|y|, \quad s\in (0,2).
\]
From \eqref{eq:diff_frac}, we have  
\[
-\Delta v -(-\Delta)^{\frac{\alpha}{2}} v = \int_\alpha^2 \log(\sqrt{-\Delta})(-\Delta)^{\frac{s}{2}} v \, \d s 
\]
for function $v$ in some appropriate class. This gives us the following key identity: 
\begin{equation}\label{eq:key}
    \langle f, \pi-\pi_\alpha \rangle=\int_{\alpha}^{2}  \left\langle A_s u,  T \pi \right\rangle \, \d s, \quad A_s = \log(\sqrt{-\Delta})  (-\Delta)^{\frac{s}{2}} T^{-1}. 
\end{equation}

We remark that the investigate of $\log$-Laplacian is a very active topic in the analysis of PDEs, and has many applications in fluid mechanics, dispersive equations, jump Markov processes, etc, see for instance \cite{CW2019dirichlet, HZ2025nonlocal}.

\medskip\noindent
\paragraph{\em Step 4: Estimate}
The final step in our analysis consists of proving the uniform boundedness of the integral term in $s$. To this end, we need to establish the growth estimates for $A_s u$ and the decay estimates for $T\pi$, which are given in Proposition \ref{prop:poisson} and Proposition \ref{prop:IM} below, respectively.

\begin{table}[H]
\centering
\begin{tabular}{c|c|c}
\hline
\textbf{Conditions on $b$} & \textbf{Properties} & \textbf{Conclusions} \\
\hline
\multirow{2}{*}{Dissipativity \eqref{eq:dispt}} 
    & \makecell[c]{Tail control of $\pi$} 
    & Lemma \ref{lem:int} \\
\cline{2-3}
    & \makecell[c]{Growth control of $u$} 
    & Proposition \ref{prop:poisson}(a) \\
\hline
\multirow{2}{*}{Growth \eqref{eq:growth}} 
    & \makecell[c]{Tail control of $T \pi$} 
    & Propositions \ref{prop:IM} \\
\cline{2-3}
    & \makecell[c]{Growth control of $A_s u$} 
    & Proposition \ref{prop:poisson}(b) \\
\hline
\end{tabular}
\end{table}

To the best of our knowledge, this approach is novel in the literature. Our result shows that this methodology yields quantitative bounds on the total variation distance between $\pi$ and $\pi_\alpha$ under very weak assumptions on the drift field $b$. Furthermore, this framework is readily extensible to a broader class of models, such as equations with non-constant diffusion coefficients or driven by more general stable processes. 

\medskip
\paragraph{\bf Notations and Conventions}
We close this section by introducing some notation and conventions. Let $\mathbb{R}^d$ denote the $d$-dimensional Euclidean space. For $x \in \mathbb{R}^d$ and $r > 0$, we denote by $B_r(x)$ the open ball of radius $r$ centered at $x$, and $B_r := B_r(0)$. For two positive quantities $A$ and $B$, the notation $A \gg B$ means that $A$ is much larger than $B$. The notation $C$ denotes a generic positive constant whose value may change from line to line. We also use the following standard notations.

\begin{itemize}
    \item $\mathscr{S}(\mathbb{R}^d)$: the Schwartz space of rapidly decreasing smooth functions.
    \item $\mathscr{S}'(\mathbb{R}^d)$: the space of tempered distributions.
    \item We denote the Hölder space of exponent $\gamma$ on an open set $\Omega \subset \mathbb{R}^d$ as $C^{\gamma}(\Omega)$. Similarly, $C_{\rm loc}^{\gamma}(\Omega)$ denotes the space of functions that are locally Hölder continuous with exponent $\gamma$ on $\Omega$. 
    \item For any $\gamma \in \mR_+ \backslash \mN$, 
    \[
        [u]_{\gamma;\Omega}:=\sup_{x,y\in \Omega}\frac{\left|\nabla^{[\gamma]}f(x)-\nabla^{[\gamma]}f(y)\right|}{|x-y|^{\gamma-[\gamma]}}.
    \]
\end{itemize}

\section{Two Auxiliary Fourier Multiplier Operators}\label{sec:T_A}

In this section, we first construct an invertible Fourier multiplier operator $T$. While similar in structure to the Bessel potential $(\mathop{I}-\Delta)^{1/3}$, our operator $T$ is characterized by a kernel that decays faster than $\mathrm{e}^{-|x|^{a}}$ as $|x| \to \infty$ for some given $a > 0$. Secondly, we studied the operator $A_s$ given in \eqref{eq:key} and establish a crucial upper bound for its integral kernel in Proposition \ref{prop:A}. 

\begin{definition}
    Given a fixed measurable function $m: \mR^d \to \mR$. The corresponding multiplier operator $T_m$ with $m$ is given by
\[
T_m [f]:= \mathcal{F}^{-1}\bigl( m  \hat f  \bigr), \quad f \in \sS (\mathbb{R}^d),
\]
where $\hat f = \sF f$ denotes the Fourier transform of $f$, and $\sF^{-1}$ is the inverse Fourier transform. If $T_m$ is bounded on $L^p(\mathbb{R}^d)$, we call $m$ an $L^p$ Fourier multiplier.
\end{definition}

\subsection{Operators $T$}\label{sec:T}

First of all, let $K$ be the distributional kernel of the operator $(I-\Delta)^{1/3}$ (with the Fourier multiplier $(1+|\xi|^2)^{1/3}$). Then 
\[
(I-\Delta)^{1/3} f(x)= f(x) + \int_{\mR^d} (f(x)-f(y)) K(x-y) \d y, \quad f\in \sS(\mR^d),
\]
and 
\[
K(x)\sim C |x|^{-d-2/3}, ~ |x|\to 0 ~\text{ and }~
K(x)\sim C \e^{-|x|} , ~ |x|\to \infty 
\]
(see, e.g., \cite{Stein1970singular}). 

Next, we construct the desired operator $T$ by applying a cutoff and scaling procedure to the kernel of $(I-\Delta)^{1/3}$.

\begin{definition}[Operators $T$]
Let $a>0$ and  $M>0$. Put $\rho_M(x) = \rho (x/M)$, where $\rho$ is a smooth, radial, non-negative function such that $\rho |_{B_1}\equiv 1$, $0<\rho (x)\leq 1$ and $\rho (x) = \e^{-|x|^a}$ for $|x| \geq 2$. Define $K_T(x) = \rho_M(x)K(x)$ and the operator $T$ by 
\[
Tf(x)=f(x)+ \int_{\mR^d} (f(x)-f(y)) K_T(x-y) \d y,  \quad f\in \sS(\mR^d). 
\] 
\end{definition}

Writing $\rho_M(x) = 1 + (\rho_M(x)-1)$, we decompose the truncated kernel:
\begin{align}\label{eq:dec-K_T}
K_T(x) = K(x) + (\rho_M(x)-1)K(x) =: K(x) + R_M(x). 
\end{align}
Notice that $\rho_M(x) \equiv 1$ for $|x| \leq M$. Therefore, the remainder $R_M(x)$ vanishes identically on the ball $B_M(0)$. Moreover, $R_M(x)$ is a globally $C^\infty$-smooth function that decays exponentially. Consequently, $R_M \in L^1(\mathbb{R}^d)$ and its Fourier transform $\widehat{R}_M \in \sS' (\mR^d)$.

Taking the Fourier transform of the decomposition \eqref{eq:dec-K_T} for $K_T$ yields that the Fourier symbol of $T$ is given by
\begin{equation}\label{eq:mT}
    m_{T}(\xi) = (1+|\xi|^2)^{1/3} + \widehat{R}_M(0) - \widehat{R}_M(\xi). 
\end{equation}

The following lemma shows that $T$ is invertible and has a rapidly decaying kernel.

\begin{lemma}\label{lem:T}
There is a constant $M=M(d) \gg 1$ such that for any $a>0$ the operator $T$ given above satisfies the following properties: 
\begin{enumerate}[(a)]
    \item The Fourier symbol of $T$ satisfies 
    \[
    m_{T}(\xi) \geq \frac{1}{2} ~ \text{ and }~ m_{T}(\xi) \sim |\xi|^{\frac{2}{3}} ~\text{ as }~ |\xi| \to \infty. 
    \]
    \item The distributional integral kernel $K_T(x)$, satisfies 
    \[
    |K_T(x)|\leq C(d) |x|^{-d-\frac{2}{3}}, ~ x\in B_1 ~ \text{ and } ~ |K_T(x)| \leq C(d, a, M) \e^{-|x|^a}, ~ x\in B_1^c.
    \]
    \item The inverse operator $T^{-1}$ is well-defined and bounded on $L^2(\mR^d)$, with a strictly positive, $C^\infty_b$ symbol.
\end{enumerate}
\end{lemma}

\begin{proof}
\textit{(a) and (c):}  Recall \eqref{eq:mT}. Since $|\widehat{R}_M(\xi)|\leq \|R_M\|_{L^1}$ for all $\xi \in \mR^d$, the leading order behavior is completely governed by the multiplier, yielding $m_{T}(\xi) \sim |\xi|^\frac{2}{3}$ as $|\xi| \to \infty$. This proves (a).

Noting that, by the definition of $R_M$, 
$$
\|R_M\|_{L^1} \leq \int_{|x| \geq M} |K(x)| \, \d x, 
$$
since $K(x)$ is exponentially decaying at infinity, we can choose $M \gg 1$ sufficiently large such that $\int_{|x| \geq M} |K(x)| \d x \leq \frac{1}{4}$. Consequently, $|\widehat{R}_M(\xi)| \leq \frac{1}{4}$ for all $\xi \in \mR^d$. Because $(1+|\xi|^2)^{1/3} \geq 1$, we get 
$$
m_{T}(\xi) \geq 1 - \frac{1}{4} - \frac{1}{4} = \frac{1}{2} > 0, \quad  \xi \in \mR^d. 
$$
Thus, the inverse symbol 
$$
\frac{1}{m_{T}(\xi)}=\left[(1+|\xi|^2)^{1/3} + \widehat{R}_M(0) - \widehat{R}_M(\xi) \right]^{-1}
$$
is globally $C^\infty_b$-smooth and
\[
|1/m_T(\xi)|\leq 2 ~\textit{ and }~  1/m_T(\xi) \sim |\xi|^{-\frac{2}{3}},  ~ |\xi|\gg1. 
\]
The operator $T^{-1}$ is well-defined and bounded on $L^2(\mR^d)$, proving (c).

\medskip\noindent
\textit{(b):} For $|x| \geq 2M$, the cutoff function takes the exact form $\rho_M(x) = \e^{-|x/M|^a}$. The kernel $K(x)$ is smooth and bounded by an exponential decay $C \e^{-|x|}$ in this region. Therefore, 
$$ 
|K_T(x)| = \e^{-|x/M|^a} |K(x)| \leq C \e^{-|x|^a}, \quad |x| \gg 1, 
$$
verifying the required decay rate.
\end{proof}

\subsection{Operators $A_s$}\label{sec:A}

Motivated by the formal analysis in Section \ref{sec:MC} and the construction of $T$ given in Section \ref{sec:T}, we introduce the following definition.

\begin{definition}[Operators $A_s$]\label{def:As}
Given $s \in [7/4, 2]$. Let $A_s$ be the operator on $\sS(\mathbb{R}^d)$ associated with the multiplier 
\[
m_{A_s}:=|\xi|^s \log |\xi| / m_T(\xi), 
\] 
where $m_T$ is the multiplier of $T$ given by \eqref{eq:mT} with the same constant $M=M(d)$ as in Lemma \ref{lem:T}.
\end{definition}

\begin{proposition}\label{prop:A}
Let $s \in [7/4, 2]$. There exists a unique, real-valued, symmetric kernel $K_{A_s} \in C^\infty(\mathbb{R}^d \setminus \{0\})$ such that $A_s$ can be represented as the principal value singular integral:
\begin{equation*}
    \begin{aligned}
        A_sf(x) = \mathrm{P.V.}\int_{\mathbb{R}^d} (f(x-y) - f(x))K_{A_s}(y) \, \d y, \quad f\in \sS(\mR^d). 
    \end{aligned}
\end{equation*}
Furthermore, $K_{A_s}(y)$ satisfies 
    \begin{equation}\label{eq:Ks}
        |K_{A_s}(y)| \leq C \left[ |y|^{-d-\beta} \wedge |y|^{-d-s} \right] (1+|\log|y||), 
    \end{equation}
    where $\beta = s - 2/3 \in [13/12,4/3]$. 
\end{proposition}

To establish Proposition \ref{prop:A}, we require the following lemma.

\begin{lemma}\label{lem:dyadic}
Let $h \in \sS'(\mR^d)$. 
\begin{enumerate}[(a)]
    \item Suppose that $h(\xi)$ is a real-valued function supported in $B_1^c$ and 
    \[
    |\nabla^k h(\xi)| \leq A |\xi|^{\beta-k} \log(1+|\xi|), \quad \exists \beta \in [13/12,4/3], ~ \forall 0\leq k\leq d+3. 
    \]
    Then there exists a kernel $H$ such that the Fourier multiplier operator $T_hf(x):=\sF^{-1}(h\widehat{f})(x)$ admits the integral representation 
    \begin{equation}\label{eq:def_H1}
        T_hf(x) = \mathrm{P.V.}\int_{\mR^d} \left(f(x-y)-f(x)\right) H(y) \d y, 
    \end{equation}
    and for any $y \in \mR^d$, it holds that 
    \begin{equation}\label{eq:H1}
        |H(y)| \leq C \left[ |y|^{-d-\beta}(1+|\log|y||) \wedge |y|^{-d-2} \right], 
    \end{equation}
    where $C$ depends only on $d$ and $A$. 
    \item Suppose that $h(\xi)$ is supported in $B_2$ and 
    \[
    |\nabla^k h(\xi)| \leq B |\xi|^{s - k} \log(1+|\xi|^{-1}), \quad \exists s\in [7/4,2],~ \forall 0\leq k\leq d+4. 
    \]
    Then there exists a kernel $H$ such that the Fourier multiplier operator $T_hf(x):=\sF^{-1}(h\widehat{f})(x)$ admits the integral representation 
    \[
    T_hf(x) = \int_{\mR^d} f(x-y) H(y) \d y, 
    \]
    and for any $y \in \mR^d$, 
    \begin{equation}\label{eq:H2}
        |H(y)| \leq C \left[ |y|^{-d-s} (1+|\log|y||) \wedge 1 \right], 
    \end{equation}
    where $C$  depends on $d$ and $B$. 
\end{enumerate}
\end{lemma}

\begin{proof}
We use the Littlewood-Paley decomposition to construct the kernel $H$ and establish the desired estimates. 

Let $\varphi_0 \in C_0^\infty(\mR^d)$ such that $\text{supp}(\varphi_0) \subseteq \{|\xi| \leq 2\}$ and $\varphi_0 \equiv 1$ on $\{|\xi| \leq 1\}$. Define $\varphi(\xi) = \varphi_0(\xi) - \varphi_0(2\xi)$, supported on $\{1/2 \leq |\xi| \leq 2\}$. We have $\sum_{j \in \mathbb{Z}} \varphi(2^{-j}\xi) = 1$ for $\xi \neq 0$. 

\medskip\noindent
\textit{(a):} Since $h\in \sS'(\mR^d)$ is supported in $|\xi| \geq 1$, we can define 
\[
H_j(y) := \frac{1}{(2\pi)^d} \int_{\mR^d} \e^{iy \cdot \xi} h(\xi) \varphi(2^{-j}\xi) \d \xi, \quad j\geq 0
\]
Integrating by parts $2N (\leq d+3)$ times and using the assumption on $h$, one sees that  
\begin{align*}
    |y|^{2N} |H_j(y)| =& \left| \frac{1}{(2\pi)^d} \int_{\mR^d} \e^{iy \cdot \xi} (-\Delta)_\xi^N \left[ h(\xi) \varphi(2^{-j}\xi) \right]  \d \xi \right| \\
    \leq & C (1+j) 2^{j(d+\beta-2N)}, 
\end{align*}
where $C$ only depends on $d$ and $A$. Therefore, 
\[ |H_j(y)| \leq C (1+j) 2^{j(d+\beta)} (2^j |y|)^{-2N}. \]
Put 
\[
H(y) := \sum_{j=0}^\infty H_j(y). 
\] 

Next, we proceed to show that $H$ satisfies \eqref{eq:H1}, followed by a verification that $T_h$ is given by \eqref{eq:def_H1}. 

To bound $H(y)=\sum_{j=0}^\infty H_j(y)$, we split the sum at integer $j_0$ where $2^{j_0} \sim |y|^{-1}\, (|y|<1)$, and use $N=0$ for $j \leq j_0$, and $d+\beta<2N\leq d+3$ for $j > j_0$: 
\begin{equation}\label{eq:H1_small}
    \begin{aligned}
        |H(y)| \leq & \sum_{j=0}^\infty |H_j(y)| \leq C\sum_{j \leq j_0} (1+j) 2^{j(d+\beta)} + C\sum_{j > j_0} (1+j) 2^{j(d+\beta)} (2^j |y|)^{-2N}\\
        \leq & C (1+j_0) 2^{j_0(d+\beta)} + |y|^{-2N} \sum_{j > j_0} (1+j_0) 2^{j(d+\beta-2N)} \\
        \leq & C |y|^{-d-\beta} (1+|\log |y||), \quad |y|<1, 
    \end{aligned}
\end{equation}
where $C$ only depends on $d$ and $A$. By same argument and choosing $2N\in \{d+2, d+3\}$, one can also see that, 
\begin{equation}\label{eq:H2_big}
    |H(y)|\leq C|y|^{-d-2}, \quad |y|\geq 1, 
\end{equation}
where $C$ only depends on $d$ and $A$. 

By \eqref{eq:H1_small} and \eqref{eq:H2_big}, the right-hand side of \eqref{eq:def_H1} is well-defined. To prove \eqref{eq:def_H1}, we take Fourier tansformation: 
\begin{align*}
    &\sF\left( \mathop{\mathrm{P.V.}} \int_{\mR^d} \left(f(x-y)-f(x)\right) H(y) \d y \right)(\xi) \\ 
    =& \widehat{f}(\xi) \int_{\mR^d} (\cos(y\cdot \xi) -1) H(y) \d y = \widehat{f}(\xi) \sum_{j\geq 0} \int_{\mR^d} (\cos(y\cdot \xi) -1) H_j(y) \d y \\
    =& \widehat{f}(\xi) \sum_{j\geq 0} ( \mathrm{Re}(h(\xi))-h(0)) \varphi(2^{-j}\xi) \\
    =& \widehat{f}(\xi) h(\xi)=\sF(T_h f)(\xi). 
\end{align*}
Here we used the fact that $h$ is real-valued and $h(0)=0$. So we obtain our desired assertion. 

\medskip\noindent
\textit{(b):} In this case $h(\xi)$ is supported in $B_2$, to align notation, we put $k = -j \geq -1$. Set 
\[
H_k(y) = \frac{1}{(2\pi)^d} \int_{\mR^d} \e^{iy \cdot \xi} h(\xi) \varphi(2^k \xi) \, \d \xi. 
\]
Following similar arguments as in (a), we have 
\begin{equation*}
    \begin{aligned}
        |y|^{2N} |H_k(y)| =& \left| \frac{1}{(2\pi)^d} \int_{\mR^d} \e^{iy \cdot \xi} (-\Delta)_\xi^N \left[ h(\xi) \varphi(2^{k}\xi) \right]  \d \xi \right| \\
        \leq &  Ck 2^{k(2N-d-s)}
    \end{aligned}
\end{equation*}
Spliting the sum $\sum_{k=0}^\infty H_k(y)$ at $k_0$ where $2^{k_0} \sim |y| > 4$, we use $d+s< 2N \leq d+4$ for $k \leq k_0$, and use $N=0$ for $k > k_0$: 
\begin{equation*}
    \begin{aligned}
        \sum_{k=0}^\infty H_k(y) \leq& C\sum_{-1\leq k \leq k_0} (2+k) 2^{-k(d+s)} (2^{-k}|y|)^{-2N} + C\sum_{k > k_0} (2+k) 2^{-k(d+s)}\\
        = & C |y|^{-2N} \sum_{k \leq k_0} (2+k) 2^{k(2N-d-s)} + C\sum_{k > k_0} (2+k) 2^{-k(d+s)}\\
        \leq & C |y|^{-2N} k_0 2^{k_0(2N-d-s)}+ C k_0 2^{-k_0(d+s)}\\
        \leq & C |y|^{-d-s}\log|y|, 
    \end{aligned}
\end{equation*}
where $C$ only depends on $d$ and $B$. Noting that $h\in L^1$, we have  
\[
|H(y)|\leq C(d, B), \quad y\in \mR^d.  
\]
Then following similar arguments as in (a), we can obtain our desired assertion.
\end{proof}

Now we are in a position to give  
\begin{proof}[Proof of Proposition \ref{prop:A}]
    Let $\chi\in C_c^\infty(B_2)$ be a cut-off function such that $\chi(x) \equiv 1$ when $x\in B_1$. We set 
    \[
    h_{I}:=(1-\chi) m_{A_s} ~ \text{ and }~ h_{O}:= \chi m_{A_s}. 
    \]
    We choose $M\gg 1$ such that 
    \[
        \sup_{|k|\leq d+4} |\nabla^k R_M| \ll 1. 
    \]
    It is then straightforward to show that $h_I$ and $h_O$ satisfy the assumptions of Lemma \ref{lem:dyadic} (a) and (b), respectively. By Lemma \ref{lem:dyadic}, there exist integral kernels $H_I$ and $H_O$ associated with $h_I$ and $h_O$ that satisfy \eqref{eq:H1} and \eqref{eq:H2}, respectively. Put $K_{A_s}:= H_I+H_O$, then we can see that $K_{A_s}$ satisfies \eqref{eq:Ks}, and for each $f\in \sS(\mR^d)$,
    \begin{align*}
        A_s f(x) = & \mathrm{P.V.} \int_{\mR^d} \left( f(x-y)-f(x) \right) K_{A_s}(y)\d y + h_O(0) f(x)\\
        = & \mathrm{P.V.} \int_{\mR^d} \left( f(x-y)-f(x) \right) K_{A_s}(y)\d y.
    \end{align*}
    The proof is finished.
\end{proof}

\subsection{Extensions of Operators}

While $T$ and $A_s$ have thus far been defined as operators on $\sS(\mathbb{R}^d)$ (see Sections \ref{sec:T} and \ref{sec:A}), the approximate estimates of their kernels allow for their extension to broader function spaces. For instance, for any $s \in [7/4, 2]$ and $0 < \eps \ll 1$, define the space
\[
\cX_{s,\eps}:= \left\{ f \in C^{\beta+\eps}_{\rm loc}(\mR^d) : \beta=s-{2}/{3}, \quad \sup_{x\in \mR^d} \frac{|f(x)|}{(1+|x|)^{s-\eps}}<\infty\right\}.
\]
One can extend operator $A_s$ given in Definition \ref{def:As} to functions $f\in \cX_{s,\eps}$, and the extended operator can be still given by the same singular integral formula:
\begin{equation}\label{eq:def_As}
    A_s f(x) = \lim_{\delta \to 0} \underbrace{\int_{|y|>\delta} (f(x-y) - f(x))K_{A_s}(y) \, \d y}_{=:A_s^\delta f(x)}, \quad f\in \cX_\eps. 
\end{equation}
Moreover, we have 
\begin{lemma}[Symmetry of $A_s$] \label{lem:dual}
Let $s \in [7/4, 2]$ and $0<\eps\ll1$. The operator $A_s$ given by \eqref{eq:def_As} satisfies 
\begin{equation*}
    \langle A_s f, g \rangle = \langle f, A_s g \rangle, \quad f \in \mathcal{X}_{s,\eps} ~\text{ and }~ g \in C_c^\infty(\mR^d). 
\end{equation*}
\end{lemma}

\begin{proof}
By definition, the duality pairing is given by 
$$
\langle f, A_s g \rangle = \int_{\mathbb{R}^d} f(x) \Big( \lim_{\delta \to 0} A_s^\delta g(x) \Big) \, \mathrm{d}x.
$$
Since $g \in C_c^\infty(\mathbb{R}^d)$, \eqref{eq:Ks} implies that $A_s^\delta g(x)$ converges pointwise to $A_s g(x)$ and satisfies the uniform decay bound 
\[
|A_s^\delta g(x)| \leq C (1+|x|)^{-d-s} \log(2+|x|).
\]
Given the growth condition on $f$, the product obeys 
\[
|f(x) A_s^\delta g(x)| \leq C (1+|x|)^{-d-\eps} \log(2+|x|), 
\]
which provides an $L^1(\mathbb{R}^d)$ majorant. Applying the Dominated Convergence Theorem, we may pull the limit outside the integral to obtain
\[
\langle f, A_s g \rangle = \lim_{\delta \to 0} \iint_{|x-y|>\delta} K_{A_s}(x-y) (g(x)-g(y)) f(x) \, \mathrm{d}y\,  \mathrm{d}x.
\]

For any fixed $\delta > 0$, the integrand is absolutely integrable on the truncated domain. Invoking Fubini's Theorem and the symmetry of the kernel $K_{A_s}$, we can re-index $x$ and $y$ to yield the symmetrized form:
\[
\langle f, A_s g \rangle =  \frac{1}{2} \lim_{\delta \to 0}\iint_{|x-y|>\delta} K_{A_s}(x-y) (g(x)-g(y)) (f(x)-f(y)) \, \mathrm{d}x \mathrm{d}y.
\]
To pass to the limit $\delta \to 0$, we establish the absolute integrability of 
\[
G(x,y) := |K_{A_s}(x-y)| |g(x)-g(y)| |f(x)-f(y)|
\]
on $\mathbb{R}^d \times \mathbb{R}^d$. 

Assume $\mathrm{supp}(g) \subset B_R$, $R\gg 1$. The $C_c^\infty$ regularity of $g$ and the local Lipschitz property of $f$ (since $\beta=s-\frac{2}{3} > 1$) imply that 
\[
\iint_{B_R \times B_R} G(x,y) \, \d x \d y\leq C \iint_{B_R \times B_R}|x-y|^{-d+2-\beta} \log(1+1/|x-y|)\, \d x \d y<\infty,
\]
since $2-\beta = \frac{8}{3}-s> 0$.

When $x\in B_R$ and $y\in B_R^c$, we have $G(x,y)\leq |K_{A_s}(x-y)|(C+|f(y)|)$. By the definition of $\cX_{s,\eps}$ and \eqref{eq:Ks}, we have 
\[
\iint_{B_R \times B_R^c} G(x,y) \, \d x \d y \leq C \int_{B_R^c} |y|^{-d-\eps} \log (1+|y|) \d y<\infty.  
\]
By symmetry, one sees that $G \in L^1(\mathbb{R}^d \times \mathbb{R}^d)$, the Dominated Convergence Theorem allows us to pass the limit inside the integral to obtain the identity
\begin{equation} \label{eq:sym}
\langle f, A_s g \rangle = \frac{1}{2} \iint_{\mathbb{R}^d \times \mathbb{R}^d} K_{A_s}(x-y) (g(x)-g(y)) (f(x)-f(y)) \mathrm{d}x \mathrm{d}y.
\end{equation}

Following the same procedure for $\langle A_s f, g \rangle$, justified by the compact support of $g$ and the uniform bounds on $A_s^\delta f$ provided by the definition of $\mathcal{X}_{s,\eps}$, we arrive at the identical symmetrized expression \eqref{eq:sym}. Thus, the operator $A_s$ is symmetric with respect to this pairing, concluding the proof.
\end{proof}

\section{Properties of Invariant Measures and Poisson Equations}\label{sec:IM_P} 

The purpose of this section is to establish the asymptotic behavior of our key components: the decay estimates for $T\pi(x)$ and the growth estimates for $A_s u(x)$ as $|x| \to \infty$. To streamline the presentation and avoid tedious technicalities, we assume $b, f \in C^\infty$ throughout the remainder of this paper. We emphasize, however, that all subsequent estimates are independent of the regularity of $b$ and $f$. 

\subsection{Invariant measure $\pi$}

First, for the reader's convenience, we provide elementary integrability estimates for the measures $\pi$ and $\pi_\alpha$. We refer to \cite{FFMP2009sharp} and \cite{BKRS2015fokker} for more delicate pointwise estimates for $\pi$.
 
\begin{lemma}\label{lem:int}
    Suppose $b$ satisfies \eqref{eq:dispt}, then
    \begin{enumerate}[(a)]
        \item for any $\lambda\in (0,\kappa/\theta)$, it holds that 
        \begin{equation}\label{eq:int_2}
        \langle \exp(\lambda |x|^\theta), \pi\rangle <\infty. 
        \end{equation}
        \item for any $\alpha>2-\theta$ and $q\in (0, \alpha+\theta-2)$, it holds that  
        \begin{equation}\label{eq:int_alpha}
            \langle |x|^{q}, \pi_\alpha \rangle <\infty. 
        \end{equation}
    \end{enumerate}
\end{lemma}
\begin{proof}
    (a) Let $V(x)=\exp[\lambda' (1+|x|^2)^{\theta/2}]$. Since 
    \[
    \nabla V(x) = \lambda' \theta x (1+|x|^2)^{\frac{\theta}{2} - 1} V(x) 
    \]
    and 
    \[
     \Delta V(x) = \lambda' \theta V(x) (1+|x|^2)^{\frac{\theta}{2} - 2} \Bigg[ \lambda' \theta |x|^2 (1+|x|^2)^{\frac{\theta}{2}} + (\theta - 2) |x|^2 + d (1+|x|^2) \Bigg],  
    \]
    one can verify that 
    \begin{equation*}
        LV\leq -c|x|^{2\theta-2} V(x) + K
    \end{equation*}
    provided that $\lambda' < \kappa/\theta $. So for any $0<\lambda<\lambda'<\kappa/\theta$, we have 
    \[
    \langle \exp(\lambda |x|^\theta), \pi\rangle\leq C \langle 1+|x|^{2\theta-2} V(x), \pi\rangle \leq C\langle LV,\pi \rangle+ C \leq C <\infty. 
    \]

\medskip\noindent
    (b) Let \begin{equation}\label{eq:Vp}
    V_p(x) = (1+|x|^2)^{p/2}, \quad p\in ((2-\theta)\vee 0, \alpha). 
    \end{equation} 
    Noting that 
    \[
        |(-\Delta)^{\alpha/2} V_p(x)|  \leq C(1+|x|^{p-\alpha})
    \]
    and 
    \[
        b(x)\cdot\nabla V_p(x) \leq -c|x|^{p+\theta-2} + C |x|^{p-2}, 
    \]
    we see that 
    \begin{equation}\label{eq:LaV}
        L_\alpha V_p (x) \leq -c |x|^{p+\theta-2} + K. 
    \end{equation}
    Following the same argument in (a), we obtain our desired assertion. 
\end{proof}
 
Next, we establish the growth control of $T \pi$ for later use in the proof of Theorem~\ref{thm:main}.

\begin{proposition}[Decay estimates of $T \pi$]\label{prop:IM}
    Suppose that $b$ satisfies \eqref{eq:dispt} and \eqref{eq:growth}, then it holds that  
    \begin{equation}
         \big| T \pi(x) \big| \leq C \exp\left( -\frac{\kappa  |x|^\theta}{10d\theta} \right), \quad x\in \mR^d, 
    \end{equation}
    where the constant $C>0$ only depends on $\Theta$. 
\end{proposition}
\begin{proof}

    Set 
    \[W(x)=\frac{\kappa}{5d\theta} (1+|x|^2)^{\frac{\theta}{2}} ~ \text{and}~ \phi=\e^{W}. 
    \]
    By our \eqref{eq:growth} and using Lemma \ref{lem:int}, one can verify that $b\in L^{d+\eps}(\pi)$, $\phi\in L^1(\pi)$ and $\nabla \phi=\nabla W \phi \in L^{4d}(\pi)$. Thanks to \cite[Theorem 3.3.1]{BKRS2015fokker}, we have 
    $v:=\phi \pi \in W^{1,4d}(\mR^d)$. Sobolev embedding then implies that for any $x\in \mR^d$, 
    \begin{equation}\label{eq:pi_holder}
        \begin{aligned}
            \|\pi\|_{C^{\frac{3}{4}}(B_1(x))} \leq& C \|v\phi^{-1}\|_{W^{1,4d}(B_1(x))}\leq  C \|\phi^{-1}\|_{W^{1,4d}(B_1(x))} \| v\|_{W^{1,4d}(B_1(x))}\\
            \leq & C \|(1+W)\e^{-W}\|_{L^\infty(B_1(x))} \| v\|_{W^{1,4d}}\\
            \leq& C \exp\left(-\frac{\kappa  |x|^\theta}{6d\theta} \right). 
        \end{aligned}
    \end{equation}
    Take $a=2\theta$ in Lemma \ref{lem:T}, using the definition of $T$ and the estimates for its kernel (Lemma \ref{lem:T} (b)) and \eqref{eq:pi_holder}, for $|x|\gg 1$, we have 
    \begin{equation*}
        \begin{aligned}
            \big| T \pi(x) \big| \leq& C|\pi(x)|+C \|\pi\|_{C^{\frac{3}{4}}(B_1(x))} \int_{B_1} |z|^{-d+\frac{1}{12}} \d z \\
            & + C \int_{B_1^c} [\pi(x)+\pi(y+z)] \e^{-|z|^{2\theta}} \d z \\
            \leq& C \exp\left(-\frac{\kappa  |x|^\theta}{6d\theta} \right) + C \int_{1\leq |z|<|x|/2} \pi(x+z) \e^{-|z|^{2\theta}} \d z \\
            &+ C \int_{|z|\geq|x|/2} \pi(x+z) \e^{-|z|^{2\theta}} \d z \leq C \exp\left(-\frac{\kappa  |x|^\theta}{10d\theta} \right).
        \end{aligned}
    \end{equation*}
\end{proof}

\subsection{Nonlocal Poisson Equations}
In this subsection, we fix
\[
p_0 : = \frac{3}{2} \vee \frac{4-\theta}{2} ~\text{ and } ~ \alpha_0:=1+p_0/2=\frac{7}{4}\vee \left(2-\frac{\theta}{4}\right).  
\]
Let 
\[
\alpha\in [\alpha_0,2) ~\text{ and }~ s\in [\alpha, 2]. 
\]
\begin{proposition}[Growth estimates of $u$ and $A_s u$]\label{prop:poisson}
    Assume $b$ satisfies \eqref{eq:dispt}, $f\in L^\infty (\mR^d)$ with $\|f\|_{L^\infty}\leq 1$. Then 
    \begin{enumerate}[(a)]
        \item There exists a solution $u$ to \eqref{eq:poisson} such that 
        \begin{equation}\label{eq:u_bd}
            |u(x)|\leq C (1+|x|^{p_0}), \quad x\in \mR^d 
        \end{equation}
        where $C$ only depends on $\Theta$. 
        \item If $b$ also satisfies \eqref{eq:growth}, then for each $s$, $u\in \cX_{s,\eps}$ and $A_s u$ given by \eqref{eq:def_As} satisfies 
        \begin{equation}
            \big| A_s u(x) \big| \leq C  \exp\left(\frac{\kappa}{20d\theta} |x|^\theta\right), \quad x\in \mR^d, 
        \end{equation}
    where $C$ only depends on $\Theta$. 
    \end{enumerate}
\end{proposition} 

\begin{proof}
    (a) By L\'evy system, it is easy to see that each compact set is petite for the process $X_t^{\alpha}$ associated with $L_\alpha$. Noting \eqref{eq:LaV}, thanks to \cite[Theorem 3.2]{GM1996liapounov}, there exists a solution $u$ to \eqref{eq:poisson} such that 
    \[
    |u(x)| \leq C (1+|x|^{p_0}), \quad x\in \mR^d, 
    \]
    where $C$ only depends on $\Theta$. 

\medskip\noindent
    (b) Put $v(x):=u(x_0+x)$. $v$ satisfies equation 
    \[
    -(-\Delta)^{\frac{\alpha}{2}} v(x) + b(x_0+x)\cdot \nabla v(x) = f(x_0+x)-\langle f, \pi_\alpha\rangle, \quad x\in B_2. 
    \]
    Our assumption \eqref{eq:growth} and  \eqref{eq:u_bd} impliy 
    \[
    \left[ \|b(x_0+\cdot)\|_{L^\infty(B_2)}\right]^{\frac{7}{3}}  \leq C \exp\left( \frac{\kappa}{30d\theta} |x_0|^\theta \right). 
    \]
    and 
    \[
    \|v\|_{L^\infty(B_2)} \leq C (1+|x_0|^{p_0}), \quad \int_{B_2^c} \frac{v(y)}{|y|^{d+\alpha}} \d y \leq C \int_{B_2^c} \frac{1+|x_0+y|^{p_0}}{|y|^{d+\alpha}} \d y \leq C (1 + |x_0|^{p_0}).  
    \]
    Using above two facts and Lemma \ref{lem:schauder} below, one can see that for any $\alpha\in [\alpha_0, 2)$, 
    \begin{equation}\label{eq:u32}
        \|u\|_{C^{\frac{3}{2}}(B_1(x_0))}=\|v\|_{C^{\frac{3}{2}}(B_1)} \leq C \exp\left(\frac{\kappa}{20d\theta} |x_0|^\theta \right), \quad x_0 \in \mR^d, 
    \end{equation}
    where $C$ only depends on $\Theta$. Noting that $p_0<s$ and $s-2/3\leq 4/3<3/2$, we have $u\in \cX_{s,\eps}$ for some $\eps>0$.
    
    Thanks to the estimates for $K_{A_s}$ from Proposition \ref{prop:A}, \eqref{eq:u_bd} and \eqref{eq:u32}, we obtain  
    \begin{equation*}
        \begin{aligned}
            \big| A_s u(x) \big| \leq& C \|u\|_{C^{\frac{3}{2}}(B_1(x))} \int_{B_1} |z|^{\frac{1}{6}-d} \log(1+1/|z|) \d z \\
            & + C \int_{B_1^c} \frac{|u(x)|+|u(x+z)|}{|z|^{d+s}} \log(1+|z|)\d z \\
            \leq& C \|u\|_{C^{\frac{3}{2}}(B_1(x))} + C \int_{B_1^c} \frac{1+|x|^{p_0} + |x+z|^{p_0}}{|z|^{d+\alpha_0}} \log(1+|z|) \d z \\ 
            \leq & C \exp\left(\frac{\kappa}{20d\theta} |x|^\theta\right). 
        \end{aligned}
    \end{equation*}
    Here we used the facts that $2\geq s \geq \alpha_0 = 1+p_0/2>p_0$ and $s-2/3\leq 4/3=3/2-1/6$. 
\end{proof}

In the proof of Proposition \ref{prop:poisson}, we utilized local regularity estimate for $v$ that are implied by Lemma \ref{lem:schauder}. To prove the latter, we require the following standard iteration lemma, which can be found in \cite{GT2001elliptic}.

\begin{lemma}\label{lem:ite}
    Let $\varphi$ be a non-negative function on $[0,1]$ such that for each $0\leq s<t\leq 1$, $\varphi$ satisfies 
    \[
    \varphi(s)\leq \theta \varphi(t) + \frac{A}{(t-s)^\beta}+B, 
    \]
    where $\theta\in (0,1)$ and $A,B, \beta \geq 0$. Then 
    \[
    \varphi(s)\leq C \left[ \frac{A}{(t-s)^\beta}+B \right], \quad 0\leq s<t\leq 1, 
    \]
    where $C$ only depends on $\theta$ and $\beta$. 
\end{lemma}

\begin{lemma}\label{lem:schauder}
    Let $\alpha\in [\alpha_0,2)$. Suppose $u$ satisfies 
    \[
    L_\alpha u = g ~\text{ in } B_2, 
    \]
    where $g\in L^\infty(B_2)$. Then 
    \begin{equation}\label{eq:inter1}
        \begin{aligned}
        \|u\|_{C^\alpha(B_{1})} \leq& C (2-\alpha)^{-1} \|g\|_{L^\infty(B_2)} \\
        &+ C (2 - \alpha)^{-\frac{\alpha}{\alpha - 1}} \left( 1+\|b\|_{L^\infty(B_2)}\right)^{\frac{\alpha}{\alpha-1}}  \|u\|_{L^\infty(B_2)}+C \int_{B_2^c} \frac{|u(y)|}{|y|^{d+\alpha}} \d y
        \end{aligned}
    \end{equation}
    and 
    \begin{equation}\label{eq:inter2}
        \begin{aligned}
        \|u\|_{C^{\frac{3}{2}}(B_{1})} \leq& C \|g\|_{L^\infty(B_2)} + C \left[ 1+\|b\|_{L^\infty(B_2)}\right]^{\frac{7}{3}}  \|u\|_{L^\infty(B_2)}+C \int_{B_2^c} \frac{|u(y)|}{|y|^{d+\alpha}} \d y, 
        \end{aligned}
    \end{equation}
    where the constant $C$ does not depend on $\alpha$. 
\end{lemma}
\begin{proof}
    
    Let $R\in (0,1]$ and $\sigma\in (0,1)$. Let $\eta$ and $\chi$ be two cutoff functions satisfying $\eta|_{B_{\sigma R}}\equiv1$, $\eta(x)=0$ for $x\notin B_{\frac{1+\sigma}{2}R}$  and $|\nabla^k\eta|\leq C(1-\sigma)^{-k}R^{-k}$;  $\chi|_{B_{\frac{3+\sigma}{4}R}}\equiv1$, $\chi(x)=0$ for $x\notin B_{\frac{7+\sigma}{8}R}$ and $|\nabla^k\chi|\leq C(1-\sigma)^{-k}R^{-k}$. 
    
    Set $v=u\eta$. Then 
    \begin{equation}\label{eq:u_loc}
    \begin{aligned}
        -(-\Delta)^{\frac{\alpha}{2}}v=&g\eta -b\cdot \nabla u \eta - [(-\Delta)^{\frac{\alpha}{2}},\eta]u \chi + \eta(-\Delta)^{\frac{\alpha}{2}} ((1-\chi)u)=:F. 
    \end{aligned}
    \end{equation}
    Let $\{\Delta_j\}_{j\in \mZ}$ be the standard Littlewood-Paley decomposition. Then $- (-\Delta)^{\frac{\alpha}{2}} \Delta_j v = \Delta_j F$. By Schauder estimates for fractional Laplacian in the whole space (see for instrance \cite{LZ2022nonlocal}), we have 
    \begin{equation}\label{eq:schauder}
        \begin{aligned}
            [u]_{\alpha; B_{\sigma R}}\leq & [v]_{\alpha} \leq \frac{C}{2-\alpha}  \sup_{j\in \mZ} 2^{\alpha j}\|\Delta_j v\|_\infty \leq  \frac{C}{2-\alpha} \|F\|_{L^\infty}\\
            \leq & \frac{C}{2-\alpha} \|g\|_{L^\infty(B_R)} + \frac{C}{2-\alpha} \|b\|_{L^\infty(B_1)} \|\nabla u\|_{L^\infty(B_R)}\\
            & + \frac{C}{2-\alpha} \left\| [(-\Delta)^{\frac{\alpha}{2}},\eta]u\chi \right\|_{L^\infty} + \frac{C}{2-\alpha} \left\|\eta(-\Delta)^{\frac{\alpha}{2}} ((1-\chi) u) \right\|_{L^\infty}. 
        \end{aligned}
    \end{equation}
    
    We claim that 
    \begin{equation}\label{eq:cmut1}
        \begin{aligned}
            I:= & \left\|[(-\Delta)^{\frac{\alpha}{2}},\eta]u\chi\right\|_{L^\infty}\\
            \leq & C (1-\sigma)^{-\alpha} R^{-\alpha} \|u\|_{L^\infty(B_R)} + C (1-\sigma)^{1-\alpha} R^{1-\alpha} \|\nabla u\|_{L^\infty(B_R)}.  
        \end{aligned}
    \end{equation}
    \begin{equation}\label{eq:cmut2}
        \begin{aligned}
            J:=& \left\|\eta(-\Delta)^{\frac{\alpha}{2}} ((1-\chi)u) \right\|_{L^\infty}\\
            \leq&  C (2-\alpha) \left[ (1-\sigma)^{-\alpha} R^{-\alpha} \|u\|_{L^\infty(B_2)} + \int_{B_2^c} \frac{|u(y)|}{|y|^{d+\alpha}} \d y \right].
        \end{aligned}
    \end{equation}
    Combining \eqref{eq:schauder}-\eqref{eq:cmut2}, we have 
    \begin{equation*}
        \begin{aligned}
            [u]_{\alpha; B_{\sigma R}} \leq& C (2-\alpha)^{-1} \|g\|_{L^\infty(B_R)}\\
            &+ C (2-\alpha)^{-1} \left( \|b\|_{L^\infty(B_1)}+(1-\sigma)^{1-\alpha} R^{1-\alpha} \right) \|\nabla u\|_{L^\infty(B_R)}\\
            & + C (2-\alpha)^{-1} (1-\sigma)^{-\alpha} R^{-\alpha} \|u\|_{L^\infty(B_R)} \\
            &+ C (1-\sigma)^{-\alpha} R^{-\alpha} \|u\|_{L^\infty(B_2)}+ C \int_{B_2^c} \frac{|u(y)|}{|y|^{d+\alpha}} \d y.
        \end{aligned}
    \end{equation*}
    By interpolation inequalities, we have
    \begin{align*}
        &C (2-\alpha)^{-1} \|b\|_{L^\infty(B_1)} \|\nabla u\|_{L^\infty(B_R)}\\
        \leq & \frac{1}{4} [u]_{\alpha; B_R} + C (2 - \alpha)^{-\frac{\alpha}{\alpha - 1}} \|b\|_{L^\infty(B_1)}^{\frac{\alpha}{\alpha-1}}  \|u\|_{L^\infty(B_R)} 
    \end{align*}
    and 
    \begin{align*}
        &C (2-\alpha)^{-1} (1-\sigma)^{1-\alpha} R^{1-\alpha} \|\nabla u\|_{L^\infty(B_R)}\\
        \leq & \frac{1}{4} [u]_{\alpha; B_R} + C (2 - \alpha)^{-\frac{\alpha}{\alpha - 1}}  (1-\sigma)^{-\alpha} R^{-\alpha} \|u\|_{L^\infty(B_R)}, 
    \end{align*}
    we obtain 
    \begin{align*}
        [u]_{\alpha; B_{\sigma R}} \leq& \frac{1}{2} [u]_{\alpha; B_R} + \frac{C}{(1-\sigma)^{\alpha} R^{\alpha}}\left[ (2 - \alpha)^{-\frac{\alpha}{\alpha - 1}}   \|u\|_{L^\infty(B_R)} + \|u\|_{L^\infty(B_2)}\right] \\
        &+C \left[ (2-\alpha)^{-1} \|g\|_{L^\infty(B_R)}+   (2 - \alpha)^{-\frac{\alpha}{\alpha - 1}} \|b\|_{L^\infty(B_1)}^{\frac{\alpha}{\alpha-1}}  \|u\|_{L^\infty(B_R)} + \int_{B_2^c} \frac{|u(y)|}{|y|^{d+\alpha}} \d y \right].  
    \end{align*}
    Set $t=R$ and $s=\sigma R$; $\varphi(t)=[u]_{\alpha;B_t}$, 
    \[
    A= C \left[ (2 - \alpha)^{-\frac{\alpha}{\alpha - 1}}  \|u\|_{L^\infty(B_1)} +  \|u\|_{L^\infty(B_2)} \right], 
    \]
    and 
    \[ 
    B=C \left[ (2-\alpha)^{-1} \|g\|_{L^\infty(B_1)}+   (2 - \alpha)^{-\frac{\alpha}{\alpha - 1}} \|b\|_{L^\infty(B_1)}^{\frac{\alpha}{\alpha-1}}  \|u\|_{L^\infty(B_1)} + \int_{B_2^c} \frac{|u(y)|}{|y|^{d+\alpha}} \d y \right].
    \] 
    By Lemma \ref{lem:ite}, one sees 
    \begin{equation*}
    \begin{aligned}
        \|u\|_{C^\alpha(B_{1/2})} \leq& C (2-\alpha)^{-1} \|g\|_{L^\infty(B_1)} + C (2 - \alpha)^{-\frac{\alpha}{\alpha - 1}} \left( 1+ \|b\|_{L^\infty(B_1)} \right)^{\frac{\alpha}{\alpha-1}}  \|u\|_{L^\infty(B_1)}\\
        &+C \|u\|_{L^\infty(B_2)}+C\int_{B_2^c} \frac{|u(y)|}{|y|^{d+\alpha}} \d y.
    \end{aligned}
    \end{equation*}
    Then by standard scaling and covering arguments, we can obtain the desired estimate \eqref{eq:inter1}. 
    
    So our task now is to show \eqref{eq:cmut1} and \eqref{eq:cmut2}. For the claim \eqref{eq:cmut1}, we have 
    \begin{equation*}
        \begin{aligned}
            I \leq& C (2-\alpha) \int_{\mR^d} \underbrace{\frac{|[u\chi(x+z)-u\chi(x)]\cdot [\eta(x+z)-\eta(x)]|}{|z|^{d+\alpha}}}_{=:\Gamma} \d z\\
            \leq& C (2-\alpha) \int_{|z|<\frac{(1-\sigma)R}{16}} \Gamma + C (2-\alpha) \int_{|z|\geq \frac{(1-\sigma)R}{16}} \Gamma =: I_{in}+I_{out}. \\ 
        \end{aligned}
    \end{equation*}
    Since 
    \begin{equation*}
        \begin{aligned}
            I_{in} \leq& C (2-\alpha) \int_0^{\frac{(1-\sigma)R}{16}} r^{1-\alpha} \d r \cdot (1-\sigma)^{-1}R^{-1} \\
            & \qquad \qquad \qquad \cdot \left[ (1-\sigma)^{-1}R^{-1} \|u\|_{L^\infty(B_R)}+ \|\nabla u\|_{L^\infty(B_R)}\right]\\
            \leq& C (1-\sigma)^{-\alpha} R^{-\alpha} \|u\|_{L^\infty(B_R)} + C (1-\sigma)^{1-\alpha} R^{1-\alpha} \|\nabla u\|_{L^\infty(B_R)}. 
        \end{aligned}
    \end{equation*}
    and 
    \begin{equation*}
        \begin{aligned}
            I_{out} \leq& C (2-\alpha) \int_{\frac{(1-\sigma)R}{16}}^{\infty} r^{-1-\alpha} \d r  ~ \|u\|_{L^\infty(B_R)}\leq C (2-\alpha) (1-\sigma)^{-\alpha} R^{-\alpha} \|u\|_{L^\infty(B_R)}. 
        \end{aligned}
    \end{equation*}
    So we have shown \eqref{eq:cmut1}. 
    
    For \eqref{eq:cmut2}. By definition, we have 
    \begin{equation*}
        \begin{aligned}
            J\leq& C (2-\alpha) \sup_{x\in B_{\frac{1+\sigma}{2}R}} \int_{\mR^d} \frac{|u(1-\chi)(y)|}{|x-y|^{d+\alpha}} \d y\\
            \leq& C (2-\alpha) \sup_{x\in B_{\frac{1+\sigma}{2}R}} \int_{|y|>\frac{3+\sigma}{4}R} \frac{|u(y)|}{|x-y|^{d+\alpha}} \d y\\
            \leq& C (2-\alpha) \left[ (1-\sigma)^{-\alpha} R^{-\alpha} \|u\|_{L^\infty(B_2)} + \int_{B_2^c} \frac{|u(y)|}{|y|^{d+\alpha}} \d y \right]. 
        \end{aligned}
    \end{equation*}
    Therefore, we have shown \eqref{eq:cmut2} and completed the proof for \eqref{eq:inter1}.

    Regarding \eqref{eq:inter2}, we again apply \eqref{eq:u_loc} along with the Schauder estimates for the fractional Laplacian in the whole space. Noting that $\alpha \geq \alpha_0 \geq \frac{7}{4} > \frac{3}{2}$, we obtain
    \[
        [u]_{\frac{3}{2}; B_{\sigma R}} \leq \|v\|_{C^{\frac{3}{2}}} \leq C \|F\|_{L^\infty}, 
    \]
    where $C$ is a constant independent of $\alpha$. Following the same arguments as above, we arrive at \eqref{eq:inter2}.
\end{proof}

\section{Proof of Theorem \ref{thm:main}}\label{sec:proof}

We are now in a position to prove our main result.

\begin{proof}[Proof of Theorem \ref{thm:main}]
    For any $f\in C^\infty_c$ with $\|f\|_{L^\infty}\leq 1$, let $u$ be the solution to PDE \eqref{eq:poisson} given by Proposition \ref{prop:poisson}, $T$ and $A_s$ be the operators defined in Section \ref{sec:T_A}. Since we assume $b, f\in C^\infty$, one can see that $\pi\in \sS(\mR^d)$ (see for instance \cite{BKRS2015fokker}). 
    
    We need to verify \eqref{eq:key} rigorously. By \eqref{eq:poisson} and the fact that $T^{-1}$ is a symmetric multiplier operator, we have 
    \begin{equation*}
    \begin{aligned}
        \langle f, \pi-\pi_\alpha \rangle=&\langle f-\langle f, \pi_\alpha \rangle, \pi \rangle = \langle  L_\alpha u, T^{-1} T \pi\rangle- \underbrace{\langle  u, (L^* T^{-1}) T \pi \rangle}_{=0}\\
        =& \langle T^{-1}(L_\alpha-L)u, T\pi \rangle = \langle [-\Delta-(-\Delta)^{\frac{\alpha}{2}}]T^{-1} u, T \pi \rangle\\
        =& \left\langle \int_{\alpha}^{2} \log(\sqrt{-\Delta})  (-\Delta)^{\frac{s}{2}} T^{-1} u  \, \d s,  T \pi \right\rangle. 
    \end{aligned}
    \end{equation*}
    In the pairing $\langle \cdot, \cdot \rangle$ above, the left-hand entry should be understood as a distribution, while the right-hand entry is to be understood as a test function. Taking $g_n\in C_c^\infty$ such that $g_n (\leq T\pi) \to T\pi$ in $\sS(\mR^d)$. By Lemma \ref{lem:dual}, we have
    \begin{align*}
        \langle \log(\sqrt{-\Delta})  (-\Delta)^{\frac{s}{2}} T^{-1} u, g_n \rangle = \langle u, A_s g_n \rangle=\langle A_s u, g_n \rangle. 
    \end{align*}
    Thanks to Propositions \ref{prop:IM}, Proposition \ref{prop:poisson} and Fatou's lemma, we obtain that  
    \begin{equation*}
    \begin{aligned}
        \langle f, \pi-\pi_\alpha \rangle \leq & \liminf_{n\to\infty} \int_\alpha^2 \int_{\mR^d}|A_s u(x)| |g_n(x)| \, \d x \, \d s \\
        \leq & C (2-\alpha) \int_{\mR^d} \exp\left(-\frac{\kappa  |x|^\theta}{20d\theta} \right) \d x\leq C(2-\alpha), 
    \end{aligned}
    \end{equation*}
    where the constant $C$ only depends on $\Theta$. Therefore, 
    \[
    \|\pi_\alpha-\pi\|_{\mathrm{TV}} =\sup_{\|f\|_{L^\infty} \leq 1} \langle f, \pi_\alpha-\pi \rangle= \sup_{f\in C^\infty_c; \|f\|_{L^\infty} \leq 1} \langle f, \pi_\alpha-\pi \rangle \leq C (2-\alpha). 
    \] 
    So we complete the proof for our main result. 
\end{proof}

\bigskip

\paragraph{\bf Acknowledgments.} We thank Xianming Liu (Huazhong University of Science and Technology) for his insightful discussions.

\bibliographystyle{alpha}
\bibliography{mybib.bib}

\end{document}